\documentclass[11pt]{article}
\usepackage{amsmath}
\usepackage{amsfonts}
\usepackage{amssymb}
\usepackage[all]{xy}
\usepackage{stmaryrd}

\newtheorem{definition}{Definition}
\newtheorem{proposition}{Proposition}

\newenvironment{proof}{{\bf Proof:}}{$\text{ }\blacksquare$}

\begin{document}
\title{What is the higher-dimensional infinitesimal groupoid of a manifold?}
\author{Dennis Borisov\\ Department of Mathematics, Yale University\\ dennis.borisov@gmail.com}
\date{\today}
\maketitle

\begin{abstract}
The construction (by Kapranov) of the space of infinitesimal paths on a manifold is extended to include higher dimensional infinitesimal objects, encoding contractions of infinitesimal loops. This full infinitesimal groupoid is shown to have the algebra of polyvector fields as its non-linear cohomology.
\end{abstract}

What is the infinitesimal version of the fundamental groupoid of a manifold $\mathcal M$? The standard answer is that it is the Lie-Rinehart algebra (also called Lie algebroid) of vector fields on $\mathcal M$. However, this answer is not precise.

\smallskip

Let $\alpha,\beta$ be two vector fields on $\mathcal M$, and assume for simplicity that they commute, i.e. $[\alpha,\beta]=0$. The left hand side of this equation represents a loop, while the right hand side stands for the constant path. By equating the two sides we actually contract a loop. 

In \cite{Ka07} another Lie-Rinehart algebra was constructed for each $\mathcal M$. Starting with the space $\mathfrak X^1$ of vector fields, one builds the free Lie-Rinehart algebra $\mathcal R(\mathfrak X^1)$, generated by $\mathfrak X^1$. The action of $\mathcal R(\mathfrak X^1)$ on functions on $\mathcal M$ is generated by the action of $\mathfrak X^1$, but otherwise the Lie bracket is free. 

The Lie bracket being free means that one doesn't contract non-degenerate loops (we still have $\llbracket\alpha,\alpha\rrbracket=\llbracket\alpha,0\rrbracket=0$), and therefore one can call $\mathcal R(\mathfrak X^1)$ {\bf the space of infinitesimal paths}. This is a module over the algebra $\mathfrak X^0$ of functions on $\mathcal M$, and it defines a vector bundle on $\mathcal M$, that we will denote by $\mathcal RT\mathcal M$. Unless $dim(\mathcal M)\leq 1$, $\mathcal RT\mathcal M$ is obviously infinite dimensional.

\smallskip

While $\mathfrak X^1$ is obtained by contracting all loops, $\mathcal R(\mathfrak X^1)$ is built by avoiding contractions. In this paper we add higher dimensional components to $\mathcal R(\mathfrak X^1)$, that represent higher dimensional submanifolds, needed to parametrize contractions. We obtain {\bf the full infinitesimal groupoid} $\mathbb X^*$, graded by the dimension of submanifolds. 

The algebraic structure on $\mathbb X^*$ is considerably more complicated than that of a Lie-Rinehart algebra. In particular there are the {\bf homotopy maps} 
	\begin{equation}\mathbb X^k\rightarrow\mathbb X^{k-1},\end{equation}
that represent contractions of loops. For example: $\mathbb X^1=\mathcal R(\mathfrak X^1)$, but $\mathbb X^1$ modulo the image of $\mathbb X^2\rightarrow\mathbb X^1$ is just $\mathfrak X^1$.

We also consider the ``non-linear cohomology'' $\mathcal H(\mathbb X^*)$ of $\mathbb X^*$ (for $k>1$ the set $\mathbb X^k$ is not additive), i.e. homotopy classes of elements of $\mathbb X^*$, that themselves do not define non-trivial equivalence relations. We obtain that $\mathcal H(\mathbb X^*)$ is the algebra of polyvector fields $\underset{k=1}{\overset{-\infty}\Sigma}\wedge^{-k+1}\mathfrak X^1$. It is important to note that this algebra is not Gerstenhaber, since the Lie bracket has degree $0$, while the wedge product is of degree $-1$.\footnote{We use the cohomological notation, i.e. differentials raise degrees.}

\smallskip

The idea of construction of $\mathbb X^*$ is as follows. A loop like $[\alpha,\beta]$ is contracted by a $2$-morphism, i.e. it happens inside a jet of a submanifold of dimension $2$. While $\alpha,\beta$ are given as $1$-jets, it is not enough to take the $1$-jet of the surface, indeed, if $\alpha,\beta$ are given as morphisms 
	$$\alpha:\mathfrak X^0\rightarrow\mathfrak X^0\otimes\mathbb R[\epsilon_1]/(\epsilon_1^2),\quad\beta:\mathfrak X^0\rightarrow\mathfrak X^0\otimes\mathbb R[\epsilon_2]/(\epsilon_2^2),$$ 
their bracket requires $\epsilon_1\epsilon_2$, and hence we need a morphism 
	$$\mathfrak X^0\rightarrow\mathfrak X^0\otimes\mathbb R[\epsilon_1,\epsilon_2]/(\epsilon_1^2,\epsilon_2^2).$$
In other words we need to consider sections of the second tangent bundle $T^2\mathcal M$. If $\nu\in\mathfrak X^2$ is such a section, it has $\alpha,\beta$ as its two projections to $\mathfrak X^1$, and it defines a homotopy relation in $\mathbb X^1$ by equating $[\alpha,\beta]=\llbracket\alpha,\beta\rrbracket$, where $\llbracket-,-\rrbracket$ is the free bracket on $\mathbb X^1=\mathcal R(\mathfrak X^1)$.

\smallskip

To understand how $\nu$ provides the identification $[\alpha,\beta]=\llbracket\alpha,\beta\rrbracket$ one should note the two ways to view $\nu$ as a one-parameter family of vector fields on $\mathcal M$. On one hand $\nu$ is tangent to $\alpha$, on the other it is tangent to $\beta$. Choosing one of the ways is equivalent to choosing an order on the pair $\{\alpha,\beta\}$, i.e. choosing an orientation on $\nu$.

The symmetric group $\mathbb S_2$ acts on the set of these choices, and if $\sigma\in\mathbb S_2$ is the non-trivial element, its action is well known to produce $[\alpha,\beta]$, indeed, let $\alpha_*(\beta),\beta_*(\alpha)$ be the sections of $T^2\mathcal M$, obtained by using functoriality of $T$. Then 
	\begin{equation}\alpha_*(\beta)-\beta_*(\alpha)=[\alpha,\beta],\end{equation}
where on the left hand side we use the {\it strong difference} of points in $T^2\mathcal M$ (\cite{KL84}, \cite{MR91}), and on the right hand side we use identification of points in $T\mathcal M$ with tangents to the fibers of $T\mathcal M\rightarrow\mathcal M$.

Having two brackets $[-,-]$ and $\llbracket-,-\rrbracket$ on $\mathcal RT\mathcal M$ we have two actions of $\mathbb S_2$ on $T(\mathcal RT\mathcal M)$, and taking their strong difference we obtain $[\alpha,\beta]-\llbracket\alpha,\beta\rrbracket$.

\smallskip 

There are many sections of $T^2\mathcal M$, that have $\alpha,\beta$ as their projections to $\mathfrak X^1$, and therefore there are many different ways to contract the loop $[\alpha,\beta]$. To equate the different ways we need to consider jets of submanifolds of dimension $3$, again these jets should be $3$-jets of a particular kind, i.e. maps to $\mathbb R[\epsilon_1,\epsilon_2,\epsilon_3]/(\epsilon_i^2)$. So we need sections of $T^3\mathcal M$.

Since the combinatorics of $\{T^k\mathcal M\}$ is not globular but cubical, a section $\mu:\mathcal M\rightarrow T^3\mathcal M$ defines several homotopy relations on $\mathbb X^2$. There are $3$ pairs of generators in $\mathbb R[\epsilon_1,\epsilon_2,\epsilon_3]/(\epsilon_i^2)$, and hence there are three homotopy maps $\mathfrak X^3\rightarrow\mathbb X^2$. For an arbitrary $k$, $\nu\in\mathfrak X^k$ defines $\frac{k!}{2(k-2)!}$ homotopies.

\smallskip

To continue this construction further we have to work with iterations of $\mathcal RT$, rather than with iterations of the usual tangent bundle, i.e. instead of $T^2\mathcal M$ we should take sections in $(\mathcal RT)^2\mathcal M$. 

One can construct $(\mathcal RT)^2\mathcal M$, but it is too big. It contains infinitesimal loops, that are completely inside the fibers of $\mathcal RT\mathcal M\rightarrow\mathcal M$. Tangents to these fibers represent infinitesimal automorphisms of tangents to $\mathcal M$, and, as far as $\mathcal M$ is concerned, infinitesimal loops in these fibers should be contracted.

\smallskip

This leads us to the construction of {\bf relatively free Lie-Rinehart algebras}. We formulate this in general terms: let $\pi:\mathcal N_1\rightarrow\mathcal N_2$ be a smooth map, that locally (on $\mathcal N_1$) is a trivial bundle (not necessarily linear). We define $\mathcal R(\mathfrak X^1(\mathcal N_1),\pi)$ to be the space of infinitesimal paths, obtained from $\mathcal R(\mathfrak X^1(\mathcal N_1))$ by contracting all ``vertical loops'' with respect to $\pi$. 

If $\pi:\mathcal N_1\rightarrow pt$ is the unique map to a point, we obtain the usual space of vector fields, if $\pi:\mathcal N_1\rightarrow\mathcal N_1$ is the identity map, we obtain the space $\mathcal R(\mathfrak X^1(\mathcal N_1))$ of all infinitesimal paths from \cite{Ka07}. 

\smallskip

Applying this construction to $\mathcal M$, and iterating, we obtain a sequence $\{\mathbb T^k\mathcal M\}$, that, just like $\{T^k\mathcal M\}$, is a semi-simplicial diagram of linear bundles. We define $\mathbb X^k$ as the set of sections $\mathcal M\rightarrow\mathbb T^k\mathcal M$.

For $k>1$, the set $\mathbb X^k$ is not additive, but it has $k$ different additions {\it over} $\mathbb X^{k-1}$. Also there is a cup product and a composition product, encoding infinitesimal automorphisms and Lie derivatives respectively. Taking non-linear cohomology as described above, and factoring out jets of degenerate submanifolds we obtain the algebra of polyvector fields.

\bigskip

Here is the structure of the paper. In section \ref{kVectors} we recall the construction and some algebraic properties of the usual iterated tangent bundles, and describe the additional structure one has on the sets of sections.

In section \ref{RelativelyFree} we give the construction of relatively free Lie-Rinehart algebras and discuss their functorial properties.

In section \ref{Groupoid} we construct the full infinitesimal groupoid $\mathbb X^*$, discuss some of the algebraic operations defined on this groupoid, including actions of symmetric groups, and compute its cohomology $\mathcal H(\mathbb X^*)$.

\smallskip

Everything in this paper is formulated for smooth real manifolds. All the statements and proofs are also valid, if one uses complex analytic manifolds instead.

\tableofcontents

\section{$k$-vectors and $k$-vector fields}\label{kVectors}

In this section we recall the basic properties of points and sections of iterated tangent bundles, in particular we describe decompositions of sections into sets of vector fields, subject to action by symmetric groups. These decompositions, and the action will be central to our treatment of full infinitesimal groupoids in section \ref{Groupoid}.

\smallskip 

Let $\mathcal M$ be a smooth manifold of dimension $n$. Let $T^k\mathcal M$, $k\geq 0$, be its $k$-th iterated tangent bundle, i.e. $T^k\mathcal M$ is the tangent bundle on $T^{k-1}\mathcal M$, and $T^0\mathcal M=\mathcal M$. 

It is well known (e.g. \cite{Wh82}, \cite{Be08}) that for each $k\geq 1$ there are $k$ vector-bundle projections $\{\pi_{k,i}:T^k\mathcal M\rightarrow T^{k-1}\mathcal M\}_{0\leq i\leq k-1}$, and $\{\pi_{k,i}\}_{k\geq 1}$ satisfy the usual equations of simplicial boundaries.

From the (semi-)simplicial properties of $\{\pi_{k,i}\}_{k\geq 1}$ and the equality 
	$$\pi_1\circ\pi_{2,0}=\pi_1\circ\pi_{2,1},$$ 
it follows easily, that for any $k\geq 1$ all possible projections $T^k\mathcal M\rightarrow\mathcal M$ are equal, and we will denote by $\mathfrak X^k$ the set of smooth sections $\mathcal M\rightarrow T^k\mathcal M$. We will call such sections {\bf $k$-vector fields}, and their values at points of $\mathcal M$ {\bf $k$-vectors}.

\smallskip

There are different ways to interpret points in $T^k\mathcal M$, and we will use the notion of $F$-equivalence, introduced in \cite{Wh82}, since it is very well suited for treatment of $k$-morphisms, defined as jets of $k$-dimensional sub-manifolds in $\mathcal M$.

Let $\nu,\nu':(\mathbb R^k,0)\rightarrow(\mathcal M,p)$ be two $k$-jets, {\bf $\nu$ is $F$-equivalent to $\nu'$} if for any function $f$ on $\mathcal M$ around $p$, and any $1\leq m\leq k$, we have
\begin{equation}
\frac{\partial^m}{\partial_{i_1}\ldots\partial_{i_m}}(f\circ\nu)|_0=
\frac{\partial^m}{\partial_{i_1}\ldots\partial_{i_m}}(f\circ\nu')|_0,
\end{equation}
if $i_j$'s are pairwise different. Obviously this equivalence relation depends on the choice of coordinate system on $\mathbb R^k$. 

\begin{proposition}(\cite{Wh82}, \cite{Be08})\label{SectorsJets}
Let $p$ be a point on $\mathcal M$. There is a bijective correspondence between $k$-vectors and $F$-equivalence classes of $k$-jets of maps $(\mathbb R^k,0)\rightarrow(\mathcal M,p)$.

This correspondence is natural in $\mathcal M$.
\end{proposition}
The semi-simplicial structure on $\{T^k\mathcal M\}_{k\geq 1}$ corresponds to the diagram of coordinate subspaces of $\mathbb R^n$. More precisely, let $\{x_i\}_{0\leq i<k}$ be the natural coordinate system on $\mathbb R^k$, then each $\nu:(\mathbb R^k,0)\rightarrow(\mathcal M,p)$ has $k$ faces:
\begin{equation}
\nu_i:(\mathbb R^{k-1},0)\rightarrow(\mathcal M,p),\qquad 0\leq i\leq k-1,
\end{equation}
with $\nu_i$ being the restriction of $\nu$ to the linear subspace of $\mathbb R^k$, given by vanishing of $x_i$.

If we choose a simplicial model for the $n$-groupoid of $\mathcal M$, i.e. if we consider $k$-simplices in $\mathcal M$ (submanifolds with corners) as $k$-morphisms, it is clear, that we can realize $\nu$ as the jet of a $k$-morphism between jets of $k-1$-morphisms $\{\nu_0,\ldots,\nu_{k-1}\}$. 

Note that $k$-morphisms are represented by $k$-jets, while $k-1$-morphisms are represented by $k-1$-jets. This is not really an inconsistency, since for a $k-1$-dimensional submanifold of $\mathcal M$, the $F$-equivalence class of its $k$-jet is completely determined by its $k-1$-jet.

\smallskip

To describe reparametrizations of $k$-vectors and $k$-fields, i.e. suitable changes of coordinates on $\mathbb R^k$, we use the dual language of morphisms between algebras of functions. Consider a sequence of Weil algebras 
	$$\mathcal W_k:=\mathbb R[\epsilon_0,\ldots,\epsilon_{k-1}]/\{\epsilon_i^2\}_{0\leq i\leq k-1},\quad k\geq 1.$$ 
For any point $p\in\mathcal M$, let $\mathfrak X^0_p$ be the stalk at $p$ of the sheaf of functions on $\mathcal M$. Then points in $T^k\mathcal M$ over $p$ correspond to morphisms of $\mathbb R$-algebras
	$$\mathfrak X^0_p\rightarrow \mathcal W_k,$$
that factor the evaluation map $\mathfrak X^0_p\rightarrow\mathbb R$. Automorphisms of $W_k$ provide reparametrizations of $k$-vectors, and lead to polynomial groups (\cite{Be08}). 

\smallskip

Among all the automorphisms we will be particularly interested in the action of the symmetric group $\mathbb S_k$, permuting generators of $\mathcal W_k$. These permutations induce an action of $\mathbb S_k$ on $T^k\mathcal M$, and this action is important to us, since choosing the order on the generators of $\mathcal W_k$ allows us to decompose any section $\mathcal M\rightarrow T^k\mathcal M$ into a set of sections $\mathcal M\rightarrow T\mathcal M$. 

This decomposition will allow us to introduce several important operations on sections of $T^k\mathcal M$, and these operations are well defined since they are $\mathbb S_k$-invariant.

\bigskip

It is well known (see e.g. \cite{Be08}) that any $k$-vector field $\nu\in\mathfrak X^k$ can be decomposed into a set $\{\alpha_\phi\}$ of $1$-vector fields, indexed by {\it non-emtpy} subsets $\phi\subseteq\{0,\ldots,k-1\}$. 

Since there are different possible ways to decompose, we discuss this here in detail, and we start with an example of $\beta\in\mathfrak X^2$.

\smallskip

The two projections $T^2\mathcal M\rightrightarrows T\mathcal M$ map $\beta$ to two $1$-vector fields $\alpha_0,\alpha_1$. The image of $\alpha_0:\mathcal M\rightarrow T\mathcal M$ is transversal to the fibers of $T\mathcal M$ over $\mathcal M$, and hence at the points of the image we have a decomposition of the tangent spaces into vertical and horizontal parts. Applying this decomposition to $\beta$ we obtain 
	\begin{equation}\label{2Example}\beta\mapsto\{\alpha_0,\alpha_1,\alpha_{01}\},\quad\alpha_0,\alpha_1,\alpha_{01}\in\mathfrak X^1,\end{equation}
where $\alpha_{01}$ is tangent to the fibers of $T\mathcal M\rightarrow\mathcal M$.

This decomposition, however, depends on the choice of $\alpha_0$ as the first projection. More precisely, there is the canonical action of the symmetric group $\mathbb S_2$ on $T^2\mathcal M$, and of course we can permute $\alpha_0$ and $\alpha_1$ in (\ref{2Example}). If $\sigma\in\mathbb S_2$ is the non-trivial element, we have
	\begin{equation}\label{2Permutation}\sigma(\beta)\mapsto\{\alpha_1,\alpha_0,\alpha_{01}+[\alpha_0,\alpha_1]\}.\end{equation}
Appearance of $[\alpha_0,\alpha_1]$ in (\ref{2Permutation}) is due to the following. Every function $f$ on $\mathcal M$ defines two functions on $T\mathcal M$: one by composition with the projection $\pi_1:T\mathcal M\rightarrow\mathcal M$, and the other is $df$. Reflecting this fact we can combine the two lifts into one {\bf total lift} $f\circ\pi_1+\epsilon df$ on $T\mathcal M$, where $\epsilon^2=0$. 

Being a collection of tangents to $T\mathcal M$, $\beta$ acts on the total lift $f\circ\pi_1+\epsilon df$, and using the decomposition into horizontal/vertical parts, provided by $\alpha_0$, we can write this action as follows:
	\begin{equation}\label{2Action}f\circ\pi_1+\epsilon df\mapsto\alpha_1(f)\circ\pi_1+\epsilon(\alpha_{01}(f)+\alpha_1\alpha_0(f))\circ\pi_1.\end{equation}
In fact, $\beta$ is determined by its first projection to $T\mathcal M$, i.e. $\alpha_0$, and its action (\ref{2Action}) on the total lifts of functions from $\mathcal M$. On the $\epsilon$-part this action is that of a second order differential operator, and given some choices, as for example the order on the pair $\alpha_0,\alpha_1$, we can extract the first order part: $\alpha_{01}$. The opposite choice produces a different extraction: $\alpha_{01}+[\alpha_0,\alpha_1]$.

\smallskip

In general, a section $\nu:\mathcal M\rightarrow T^k\mathcal M$ is uniquely determined by $2^k-1$ sections $\{\alpha_\phi:\mathcal M\rightarrow T\mathcal M\}$, if we fix an order on $\phi$'s. We use the lexicographical one. 

Then a set $\{\alpha_\phi\}$ of $1$-vector fields uniquely determines a $k$-vector field: for each $\phi$ of size $m$, the corresponding differential operator of order $m$ is
	\begin{equation}\underset{\phi_i>\ldots>\phi_1}\Sigma\alpha_{\phi_i}\circ\ldots\circ\alpha_{\phi_1},\end{equation}
where the sum is taken over all decompositions $\phi=\underset{1\leq j\leq i}\bigcup\phi_j$ into pairwise disjoint subsets. The action of the symmetric group $\mathbb S_k$ on $T^k\mathcal M$ is expressed as follows: let $0\leq i< k-1$, and let $\sigma_{i,i+1}\in\mathbb S_k$ be the swapping of $i$ and $i+1$. Then $\sigma_{i,i+1}(\nu_k)$ is given by $\{\alpha'_\phi\}$, where $\alpha'_\phi=\alpha_\psi$, if $\sigma_{i,i+1}(\phi)=\psi$, $\phi\neq\psi$, and for the rest of $\phi\subseteq\{0,\ldots,{k-1}\}$
	\begin{equation}\label{TheAction}\alpha'_\phi=\alpha_\phi+\underset{\phi'<\phi''}\Sigma[\alpha_{\phi'},\alpha_{\phi''}],\end{equation}
where the sum is taken over all decompositions $\phi=\phi'\cup\phi''$, s.t. $\phi'\cap\phi''=\emptyset$, and $\sigma_{i,i+1}(\phi')>\sigma_{i,i+1}(\phi'')$.

\bigskip

There are operations on $k$-vector fields that are performed pointwise, i.e. they can be defined also for $k$-vectors, and there are operations that require sections. Now we describe some of the both types of these operations in terms of $k$-vector fields.

\smallskip

We have already mentioned that for each $k\geq 1$ there are $k$ vector bundle structures on $T^k\mathcal M$ over $T^{k-1}\mathcal M$. Obviously these $k$ additions on $T^k\mathcal M$ translate to $k$ additions on $\mathfrak X^k$: two sections $\{\mu_\phi\}$, $\{\nu_\phi\}$ can be added if there is a $k-1$-dimensional face $\psi\subset\{0,\ldots,k-1\}$, i.e. $\psi$ has exactly $k-1$ elements, s.t.
	\begin{equation}\mu_\phi=\nu_\phi,\quad\forall\phi\subseteq\psi.\end{equation}
Then $\mu+_\psi\nu$ is given as follows: if $\phi\subseteq\psi$, then $(\mu+_\psi\nu)_\phi=\mu_\phi=\nu_\phi$, if $\phi\nsubseteq\psi$ then $(\mu+_\psi\nu)_\phi=\mu_\phi+\nu_\phi$.

\smallskip

An operation, closely related to the additions, is the {\bf strong difference} between two $2$-vectors. It is not defined for every couple of $\beta_1,\beta_2\in T^2\mathcal M$, but only for those that have same projections to $T\mathcal M$, i.e.
	\begin{equation}\pi_{2,0}(\beta_1)=\pi_{2,0}(\beta_2),\quad\pi_{2,1}(\beta_1)=\pi_{2,1}(\beta_2).\end{equation}
In this case it is easy to see that the difference between $\beta_1,\beta_2$ as vectors on $\pi_{2,0}(\beta_1)\in T\mathcal M$ is a vector tangent to the fiber of $\pi_{2,0}(\beta_1)$ over $\mathcal M$. This vector can be obtained by vertical lift of some $\alpha\in T\mathcal M$, which is by definition the strong difference of $\beta_1,\beta_2$  
	\begin{equation}\alpha:=\beta_1-\beta_2.\end{equation}
Clearly, one can take strong difference of a pair of $k$-vectors for any $k\geq 2$, considered as $2$-vectors on $T^{k-2}\mathcal M$. This operation is defined also for $k$-fields, and in terms of decompositions into $1$-fields it is represented as follows. 

Let $\mu=\{\alpha_\phi\}$, $\nu=\{\beta_\phi\}$ be two $k$-fields. We can define $\mu-\nu$ only if all $k-1$-faces of $\mu,\nu$ coincide, i.e. for any $\phi\subset\{0,\ldots,k-1\}$ of of size $\leq k-1$, we have $\alpha_\phi=\beta_\phi$. Then $\mu-\nu=\{\gamma_\psi\}$ is the $k-1$-field defined as follows: for any $\psi\subseteq\{0,\ldots,k-3\}$
	\begin{equation}\label{Strong}\gamma_\psi=\alpha_\psi=\beta_\psi,\quad\gamma_{\psi\cup\{k-2\}}=\alpha_{\psi\cup\{k-2,k-1\}}-\beta_{\psi\cup\{k-2,k-1\}}.\end{equation}

\smallskip

There is another important pointwise operation: {\bf the cup product}, however, it is only partially defined. Consider the following Weil algebras
	$$\mathcal V_k:=\mathbb R[\varepsilon_1,\ldots,\varepsilon_k]/(\varepsilon_i\varepsilon_j).$$
For $k>1$ $\mathcal V_k\neq\mathcal W_k$, and there are several $\mathbb R$-algebra morphisms $\mathcal V_k\rightarrow\mathcal W_k$.

Now let $p$ be a point in $\mathcal M$, and let $x:\mathfrak X^0_p\rightarrow\mathcal W_k$, $y:\mathfrak X^0_p\rightarrow\mathcal W_m$ be a $k$-vector and an $m$-vector at $p$. Suppose that $y$ factors through some $g:\mathcal V_m\rightarrow\mathcal W_m$. Then we can define a $k+m$-vector $x\cup y$, using the following $\mathbb R$-algebra morphism
	\begin{equation}\mathcal W_k\prod_\mathbb R\mathcal V_m\rightarrow\mathcal W_{k+m},\quad\epsilon_i\mapsto\epsilon_i,\quad\varepsilon_j\mapsto h\circ g(\varepsilon_j)\epsilon_0\ldots\epsilon_{k-1},\end{equation}
where $h:\mathcal W_m\rightarrow\mathcal W_{k+m}$ maps $\epsilon_i$ to $\epsilon_{k+i}$. 

For example: if $m=k=1$, then $x\cup y$ is the evaluation at $x$ of the well known ``vertical lift'' of $y$ to $T\mathcal M$. Similarly there are vertical lifts of $1$-vectors to $T^k\mathcal M$ for $k>1$. If $m>1$, this operation is not everywhere defined anymore, but only for those $y:\mathfrak X^0_p\rightarrow\mathcal W_m$, that factor through $\mathcal V_m$.

Notice that the cup product is invariant with respect to the action of symmetric groups, i.e. if $\sigma_k\in\mathbb S_k$, $\sigma_m\in\mathbb S_m$, then
	\begin{equation}(\sigma_k x)\cup(\sigma_m y)=(\sigma_k\times\sigma_m)(x\cup y).\end{equation}
	
In terms of sequences of $1$-vector fields cup product is represented as follows: let $\{\alpha_\phi\}=\mu\in\mathfrak X^k$ and $\{\beta_\psi\}=\nu\in\mathfrak X^m$, and suppose their cup product $\{\gamma_\chi\}=\mu\cup\nu\in\mathfrak X^{k+m}$ is everywhere defined (i.e. $\nu:\mathfrak X^0\rightarrow\mathcal W_m\otimes_\mathbb R\mathfrak X^0$ factors through $\mathcal V_m\otimes_\mathbb R\mathfrak X^0$), then
	\begin{equation}\gamma_\phi=\alpha_\phi,\quad\gamma_{\phi\cup\psi}=\beta_\psi,\end{equation}
and the rest of components are $0$. Here we consider $\phi\subseteq\{0,\ldots,k-1\}$ and $\psi\subseteq\{0,\ldots,m-1\}$ as subsets of $\{0,\ldots,k+m-1\}$ given by the (lexicographical) order preserving bijection
	\begin{equation}\label{OrdinalSum}\{0,\ldots,k-1\}\coprod\{0,\ldots,m-1\}\overset{\simeq}\rightarrow\{0,\ldots,k+m-1\}.\end{equation}

\smallskip

An important operation, that is not defined pointwise, is {\bf the composition product} of vector fields. Let $\mu\in\mathfrak X^k$, $\nu\in\mathfrak X^m$. Since $\mu$ is a map $\mathcal M\rightarrow T^k\mathcal M$, applying the tangent functor $m$ times we obtain a map $\mu_*:T^m\mathcal M\rightarrow T^{m+k}\mathcal M$. Evaluating at $\nu:\mathcal M\rightarrow T^m\mathcal M$ we get a section $\mu\times\nu:\mathcal M\rightarrow T^{m+k}\mathcal M$. 

In terms of sequences of $1$-vector fields this operation is written as follows. Let $\{\alpha_\phi\}$, $\{\beta_\psi\}$, $\{\gamma_\chi\}$ be $\mu,\nu,\mu\times\nu$ respectively. Then
	\begin{equation}\gamma_\phi=\alpha_\phi,\quad\gamma_\psi=\beta_\psi,\end{equation}
and the rest of components are $0$. Here again we consider $\phi,\psi$ as subsets of $\{1,\ldots,k+m\}$ using (\ref{OrdinalSum}).

Notice that also composition product is invariant with respect to the action of symmetric groups, i.e. for any $\sigma_k\in\mathbb S_k$, $\sigma_m\in\mathbb S_m$ we have
	\begin{equation}(\sigma_k\mu)\times(\sigma_m\nu)=(\sigma_k\times\sigma_m)(\mu\times\nu).\end{equation}
	
\section{Relatively free Lie-Rinehart algebras}\label{RelativelyFree}

In this section we recall (from \cite{Ka07}) the construction of the bundle of infinitesimal paths on a manifold, and relativize it, i.e. construct bundles of infinitesimal paths, where we contract some of the loops. This relative version has nice functorial properties, and we use them in the next section to build the iterated bundles of infinitesimal paths on a manifold. 

\smallskip

In \cite{Ka07} the space of infinitesimal paths on $\mathcal M$ is defined as the free Lie-Rinehart algebra, generated by $\mathfrak X^1$. Here are the details of this free construction in the general case.

Let $A$ be an $\mathbb R$-algebra, and let $M$ be an $A$-module {\bf with an anchor}, i.e. there is an $A$-map $M\rightarrow Der_{\mathbb R}(A)$. Let $\mathcal L(M)$ be the free Lie algebra over $\mathbb R$, generated by $M$. By its construction $\mathcal L(M)$ is a graded space, with $\mathcal L(M)_d$ being generated by Lie monomials of length $d$. There is an obvious extension of the anchor $\mathcal L(M)\rightarrow Der_{\mathbb R}(A)$.

The free Lie-Rinehart algebra $\mathcal R(M)$, generated by $M$, is a filtered $A$-module, inductively defined as follows:
\begin{itemize}
	\item $\mathcal R(M)_{\leq 1}=M$,
	\item for $n>1$ $\mathcal R(M)_{\leq d}$ is $\underset{1\leq i\leq d}\bigoplus\mathcal L(M)_i$ modulo the following relations
			\begin{equation}[x,fy]-[fx,y]=x(f)y+y(f)x,\quad[x,q]=0,\end{equation}
		where $x,y\in\mathcal L(M)$, $f\in A$, $q$ is in the kernel of 
			$$\underset{1\leq i\leq d-1}\bigoplus\mathcal L(M)_i\rightarrow\mathcal R(M)_{\leq d-1},$$ 
		and for any $x\in\mathcal L(M)$ we write $x(f)$ for the action of $x$ on $f$ through the anchor.
\end{itemize}
It is easy to see that the kernel of $\mathcal L(M)\rightarrow\mathcal R(M)$ is a Lie ideal, and hence $\mathcal R(M)$ inherits a Lie structure. Also it is easy to check that the anchor $\mathcal L(M)\rightarrow Der_{\mathbb R}(A)$ vanishes on the kernel of $\mathcal L(M)\rightarrow\mathcal R(M)$, and hence there is a well defined action of $\mathcal R(M)$ on $A$. Finally, the action of $A$ on $\mathcal R(M)$ is given by
	\begin{equation}f[x,y]=[fx,y]+y(f)x=[x,fy]-x(f)y.\end{equation}
	
\smallskip

Applying the above construction to the set $\mathfrak X^1$ of vector fields on $\mathcal M$, we get a new Lie-Rinehart algebra $\mathcal R(\mathfrak X^1)$, which is the space of infinitesimal paths, without contractions of surfaces, except for the degenerate ones (i.e. we do have $[\alpha,\alpha]=0$).

The sheaf $\mathcal R(\mathfrak X^1)$ is locally free, and hence we can form the linear bundle $\mathcal RT\mathcal M$, having $\mathcal R(\mathfrak X^1)$ as the set of sections. We would like to iterate this construction, i.e. we would like to consider spaces of infinitesimal paths on spaces of infinitesimal paths on $\mathcal M$, and so on. First we need to establish some of the functorial properties of $\mathcal R$.

\begin{proposition}\label{Embeddings} Let $\mathcal M$, $\mathcal N$ be two manifolds, and let $F:\mathcal M\rightarrow\mathcal N$ be a smooth map, that locally (on $\mathcal M$) is an embedding. Then $F$ extends to a map of pairs
	\begin{equation}\xymatrix{\mathcal RT\mathcal M\ar[d]\ar[rr]^{\mathcal RF} && \mathcal RT\mathcal N\ar[d]\\ \mathcal M\ar[rr]_F && \mathcal N}\end{equation}
and this extension is functorial in $F$.
\end{proposition}
\begin{proof} Since construction of $\mathcal RT$ can be done locally (\cite{Ka07}), we can (choosing local coordinates) assume that $\mathcal M=\mathbb R^m$, $\mathcal N=\mathbb R^n=\mathbb R^m\times\mathbb R^k$, and $F:\mathbb R^m\rightarrow\mathbb R^m\times\mathbb R^k$ is inclusion of a coordinate subspace.

Using the natural flat structure on $\mathbb R^n$ we have a map 
	\begin{equation}\label{Fields}F_*:\mathfrak X^1(\mathbb R^m)\rightarrow\mathfrak X^1(\mathbb R^n),\end{equation}
that is a morphism of Lie algebras. Consequently there is an induced morphism 
	$$\mathcal L(F_*):\mathcal L(\mathfrak X^1(\mathbb R^m))\rightarrow\mathcal L(\mathfrak X^1(\mathbb R^n)).$$

There is also a projection $\mathbb R^n=\mathbb R^m\times\mathbb R^k\rightarrow\mathbb R^m$, and hence we have an inclusion of algebras of functions 
	\begin{equation}\label{Functions}\mathfrak X^0(\mathbb R^m)\rightarrow\mathfrak X^0(\mathbb R^n).\end{equation} 
Using this inclusion it is easy to see that $\mathcal L(F_*)$ maps kernel of $\mathcal L(\mathfrak X^1(\mathbb R^m))\rightarrow\mathcal R(\mathfrak X^1(\mathbb R^m))$ to the kernel of $\mathcal L(\mathfrak X^1(\mathbb R^m))\rightarrow\mathcal R(\mathfrak X^1(\mathbb R^m))$, and hence we have an $\mathbb R$-linear map 
	$$\mathcal R(F_*):\mathcal R(\mathfrak X^1(\mathbb R^m))\rightarrow\mathcal R(\mathfrak X^1(\mathbb R^n)).$$
This map is also $\mathfrak X^0(\mathbb R^m)$-linear, where we see $\mathcal R(\mathfrak X^1(\mathbb R^n))$ as an $\mathfrak X^0(\mathbb R^m)$-module through (\ref{Functions}).

Now we compose $\mathcal R(F_*)$ with the projection 
	$$\mathcal R(\mathfrak X^1(\mathbb R^n))\rightarrow\mathcal R(\mathfrak X^1(\mathbb R^n))\otimes_{\mathfrak X^0(\mathbb R^n)}\mathfrak X^0(\mathbb R^m),$$
and obtain a morphism of bundles
	\begin{equation}\mathcal RF:\mathcal RT\mathcal M\rightarrow F^*(\mathcal RT\mathcal N).\end{equation}
	
We claim that $\mathcal RF$ is independent of the choice of local coordinates on $\mathcal N$. Indeed, a different choice produces different maps in (\ref{Fields}), (\ref{Functions}), but they become the same, when restricted to the image of $\mathcal M$ in $\mathcal N$, and it is easy to check that this implies the resulting $\mathcal RF$'s are equal.

\smallskip

Having two smooth maps $\mathcal M\rightarrow\mathcal N\rightarrow\mathcal N'$, s.t. each one is a local embedding, it is clear that locally we can present them as inclusions of coordinate subspaces $\mathbb R^m\rightarrow\mathbb R^{m+k}\rightarrow\mathbb R^{m+k+k'}$, and the corresponding choice of local coordinates implies functoriality.\end{proof}

\smallskip

Now we would like to iterate the $\mathcal RT$ construction. In some of our arguments we assume that the manifolds in question are finite dimensional, while $\mathcal RT\mathcal M$ is infinite dimensional, if dimension of $\mathcal M$ is greater than $1$. However, since $\mathcal RT\mathcal M$ is filtered, and each $\mathcal RT\mathcal M_{\leq d}$ is finite dimensional, we can treat $\mathcal RT\mathcal M$ as if it is finite dimensional itself, as long as everything that we do happens in some $\mathcal RT\mathcal M_{\leq d}$ for $d$ large enough.

We are interested in particular in tangent vectors to $\mathcal RT\mathcal M$, and we will always assume that each vector is tangent to some $\mathcal RT\mathcal M_{\leq d}$. Therefore, as long as there are only finitely many tangents involved, there is a finite dimensional manifold $\mathcal RT\mathcal M_{\leq d}$, where these tangents live.

\smallskip

Let $\mathfrak X^1(\mathcal RT\mathcal M)$ be the space of vector fields on $\mathcal RT\mathcal M$, s.t. for every field $\alpha$ there is $d<\infty$, s.t. $\alpha$ is tangent to $\mathcal RT\mathcal M_{\leq d}$. Applying $\mathcal R$ for each $d$, and using functoriality of $\mathcal R$ with respect to local embeddings, we obtain a bundle $(\mathcal RT)^2\mathcal M$ of infinitesimal paths on $\mathcal RT\mathcal M$.

Iterating this procedure further, we get a sequence of (filtered infinite dimensional) manifolds $\{(\mathcal RT)^k\mathcal M\}_{k\geq 1}$. However, this sequence is not the right substitute for the (semi-simplicial) sequence $\{T^k\mathcal M\}_{k\geq 1}$ of iterated tangent bundles.

\smallskip

For $k\geq 2$ $(\mathcal RT)^k\mathcal M$ is too big. For example, there are loops in $(\mathcal RT)^2\mathcal M$ that are built of vector fields, tangent to the fibers of $\mathcal RT\mathcal M\rightarrow\mathcal M$. These fibers are linear spaces and have a natural flat connection. As far as paths and surfaces on $\mathcal M$ are concerned, we are interested only in flat vertical fields on $\mathcal RT\mathcal M$, and the corresponding loops have unique flat fillings.

\smallskip

All this forces us to introduces a relative version of the free Lie-Rinehart algebra construction. Instead of anchored modules we have the following.

\begin{definition} Let $A$ be a commutative $\mathbb R$-algebra, and let $(\mathfrak g\overset{a}\rightarrow Der_{\mathbb R}(A))$ be a Lie-Rinehart algebra. {\bf An anchored $(\mathfrak g,A)$-module} is an $A$-module $M$, together with $A$-maps 
	$$\mathfrak g\overset\iota\rightarrow M\overset{b}\rightarrow Der_{\mathbb R}(A),$$
s.t. $\iota$ is injective, $a=b\iota$, and having a Lie module structure $\mathfrak g\otimes_{\mathbb R}M\rightarrow M$, s.t.
	\begin{equation}\gamma(fm)=\gamma(f)m+f\gamma(m),\quad(f\gamma)(m)=f\gamma(m)-m(f)\iota(\gamma),\end{equation}
	\begin{equation}\label{BracketComp}b(\gamma(m))=[a(\gamma),b(m)].\end{equation}	
\end{definition}
An example of an anchored $(\mathfrak g,A)$-module is given by a morphism of Lie-Rinehart algebras $\mathfrak g\rightarrow\mathfrak h$ over $A$, where we take $\mathfrak h$ to be $M$. 

If we fix $\mathfrak g$ and $A$ we have a forgetful functor from the category of Lie-Rinehart algebras under $\mathfrak g$ to the category of anchored $(\mathfrak g,A)$-modules. This functor has a left adjoint, that we now describe.

\smallskip

Let $M$ be an anchored $(\mathfrak g,A)$-module. Recall that $\mathcal L(M)$ is the free Lie algebra over $\mathbb R$, generated by $M$. The action of $\mathfrak g$ on $M$ extends to an action on $\mathcal L(M)$, by requiring that $\gamma([x,y])=[\gamma(x),y]+[x,\gamma(y)]$. 

We inductively define $\mathcal R(M,\mathfrak g)$ as follows:
\begin{itemize}
	\item $\mathcal R(M,\mathfrak g)_{\leq 1}=M$,
	\item for $d>1$ $\mathcal R(M,\mathfrak g)_{\leq d}$ is $\underset{1\leq i\leq d}\bigoplus\mathcal L(M)_i$ modulo the following relations
			\begin{equation}\label{Old}[x,fy]-[fx,y]=x(f)y+y(f)x,\quad[x,q]=0,\end{equation}
			\begin{equation}\label{New}[\iota(\gamma),x]=\gamma(x),\end{equation}
		where $x,y\in\mathcal L(M)$, $f\in A$, $\gamma\in\mathfrak g$, $q$ is in the kernel of 
			$$\underset{1\leq i\leq n-1}\bigoplus\mathcal L(M)_i\rightarrow\mathcal R(M)_{\leq n-1},$$ 
		and for any $x\in\mathcal L(M)$ we write $x(f)$ for the action of $x$ on $f$ through the anchor.
\end{itemize}
Here we consider $\mathcal L(M)$ as an $\mathbb R$-space, and divide it by the subspace, generated by (\ref{Old}), (\ref{New}). From the construction it is clear that $\mathcal R(M,\mathfrak g)$ is a Lie algebra over $\mathbb R$, inheriting the Lie structure from $\mathcal L(M)$. We claim that in addition $\mathcal R(M,\mathfrak g)$ is an $A$-module, and the projection $\mathcal L(M)\rightarrow\mathcal R(M,\mathfrak g)$ factors through $\mathcal R(M)$.

Notice that the action of $\mathfrak g$ on $\mathcal L(M)$ is compatible with (\ref{Old}), i.e. the kernel of $\mathcal L(M)\rightarrow\mathcal R(M)$ is stable under the action of $\mathfrak g$. Therefore this action extends to $\mathcal R(M)$, and hence (\ref{New}) are well defined on $\mathcal R(M)$. This implies that $\mathcal L(M)\rightarrow\mathcal R(M,\mathfrak g)$ factors through $\mathcal R(M)$.

There is an $A$-module structure on $\mathcal R(M)$, given by $f[x,y]=[fx,y]+y(f)x=[x,fy]-x(f)y$. We claim that the kernel of $\mathcal R(M)\rightarrow\mathcal R(M,\mathfrak g)$ is an $A$-submodule. Indeed, we have
	\begin{equation}f([\iota(\gamma),x]-\gamma(x))=[\iota(\gamma),fx]-\gamma(fx),\end{equation}
for any $x\in\mathcal L(M)$, $f\in A$, and $\gamma\in\mathfrak g$. Consequently $\mathcal R(M,\mathfrak g)$ is an $A$-module.

Finally we note that $\mathcal R(M,\mathfrak g)$ inherits an action on $A$ from $\mathcal R(M)$. This is rather obvious, since elements of the kernel of $\mathcal R(M)\rightarrow\mathcal R(M,\mathfrak g)$ act trivially on $A$ (this is a consequence of (\ref{BracketComp})).

\smallskip

For us the main applications of the relatively free Lie-Rinehart algebra construction are for {\bf locally trivial bundles}, i.e. smooth maps $F:\mathcal M\rightarrow\mathcal N$, s.t. locally (on $\mathcal M$) $F$ is a trivial bundle. Let $\mathfrak X^1(F)$ be the vector fields tangent to the fibers of $F$, then we have the relatively free Lie-Rinehart algebra $\mathcal R(\mathfrak X^1(\mathcal M),\mathfrak X^1(F))$ over $\mathfrak X^0(\mathcal M)$, that defines the bundle $\mathcal RT_\mathcal N\mathcal M$ of {\bf $F$-horizontal infinitesimal paths} on $\mathcal M$.

If we choose $F$ to be the unique map $\mathcal M\rightarrow pt$, we have $\mathcal RT_{pt}\mathcal M=T\mathcal M$, the usual tangent bundle. If we choose $F$ to be the identity map $\mathcal M=\mathcal M$, we have $\mathcal RT_\mathcal M\mathcal M=\mathcal RT\mathcal M$, the space of infinitesimal paths from \cite{Ka07}.

\smallskip

A nice property of $\mathcal RT_\mathcal N\mathcal M$, that will be used in the next section, is that it is functorial in both arguments, as the following proposition shows.

\begin{proposition} Let $F_1:\mathcal M_1\rightarrow\mathcal N_1$, $F_2:\mathcal M_2\rightarrow\mathcal N_2$ be two locally (on domains) trivial bundles. Suppose we are given a commutative diagram of smooth maps
	\begin{equation}\label{Diagram}\xymatrix{\mathcal M_1\ar[d]\ar[rr] && \mathcal M_2\ar[d]\\ \mathcal N_1\ar[rr] && \mathcal N_2}\end{equation}
s.t. also the horizontal arrows are locally trivial bundles. Then we have a smooth map
	\begin{equation}\label{Map}\mathcal RT_{\mathcal N_1}\mathcal M_1\rightarrow\mathcal RT_{\mathcal N_2}\mathcal M_2,\end{equation}
and (\ref{Map}) is functorial in (\ref{Diagram}).\end{proposition}
\begin{proof} First we prove functoriality in the second variable, i.e. consider the diagram
	\begin{equation}\label{Triangle}\xymatrix{&\mathcal M_1\ar[ld]_{F_1}\ar[rd]^{F_2} &\\ \mathcal N_1\ar[rr] && \mathcal N_2}\end{equation}
It is clear that $\mathfrak X^1(F_1)<\mathfrak X^1(F_2)<\mathfrak X^1(\mathcal M_1)$ as Lie algebras, and therefore it is easy to check that there is a canonical surjective morphism of Lie-Rinehart algebras
	$$\mathcal R(\mathfrak X^1(\mathcal M_1),\mathfrak X^1(F_1))\rightarrow\mathcal R(\mathfrak X^1(\mathcal M_1),\mathfrak X^1(F_2)),$$
and hence a morphism of bundles $\mathcal RT_{\mathcal N_1}\mathcal M_1\rightarrow\mathcal RT_{\mathcal N_2}\mathcal M_1$, that is obviously functorial in (\ref{Triangle}).

\smallskip

Now we prove functoriality in the first variable. Choosing local coordinates we can represent $F_1$ as the projection $\mathbb R^{k+n}=\mathbb R^k\times\mathbb R^n\rightarrow\mathbb R^n$, and  then every element of $\mathcal R(\mathfrak X^1(\mathcal M_1),\mathfrak X^1(F_1))$ can be written as follows:
	\begin{equation}\underset{i=1}{\overset{\infty}\Sigma}f_{j_1,\ldots,j_i}[\partial_{j_1},[\ldots[\partial_{j_{i-1}},\partial_{j_i}]\ldots],\end{equation}
where the sum is finite, and $i>1$ implies that $f_{j_1,\ldots,j_i}=0$ if at least one of $\partial_j$'s is tangent to the fibers of $F_1$.

Consider the following diagram
	\begin{equation}\label{FirstVar}\xymatrix{\mathcal M_1\ar[rr]\ar[rd]_{F_1} && \mathcal M_2\ar[ld]^{F_2}\\ & \mathcal N_2 &}\end{equation}
Since all maps are locally trivial bundles, we can choose local coordinates in a compatible way, i.e. locally (\ref{FirstVar}) becomes 
	\begin{equation}\xymatrix{\mathbb R^{n+m_1+m_2}\ar[rr]\ar[rd]_{F_1} && \mathbb R^{n+m_2}\ar[ld]^{F_2}\\ & \mathbb R^n &}\end{equation}
and it is clear how to define the map $\mathcal RT_{\mathcal N_2}\mathcal M_1\rightarrow\mathcal RT_{\mathcal N_2}\mathcal M_2$ locally. Functoriality and independence of the choice of coordinates are easy to check.\end{proof}

\section{The full infinitesimal groupoid}\label{Groupoid}

In this section we use the relatively free Lie-Rinehart construction of the previous section to define the sequence $\{\mathbb T^k\mathcal M\}_{k\geq 0}$ of iterated bundles of spaces of infinitesimal paths on a manifold $\mathcal M$. In particular we get a semi-simplicial structure on $\{\mathbb T^k\mathcal M\}_{k\geq 1}$. 

For $k\geq 1$ we define $\mathbb X^k$ to be the set of sections $\mathcal M\rightarrow\mathbb T^k\mathcal M$, and $\mathbb X^0:=\mathfrak X^0(\mathcal M)$. We show how every $\nu\in\mathbb X^k$, $k\geq 1$, can be decomposed into a sequence $\{\alpha_\phi\}$, with $\alpha_\phi\in\mathbb X^1$, and $\phi$ running over non-empty subsets of $\{0,\ldots,k-1\}$. This allows us to define a rich algebraic structure on $\mathbb X^*:=\{\mathbb X^k\}_{k\geq 0}$, and we call it {\it the full infinitesimal groupoid} of $\mathcal M$.

Finally we show that the non-linear cohomology of $\mathbb X^*$ is the algebra of polyvector fields on $\mathcal M$.

\smallskip

\begin{definition} Let $\mathcal M$ be a manifold. We define a sequence of (in general non-linear) locally trivial bundles $\{\pi_k:\mathbb T^k\mathcal M\rightarrow\mathcal M\}_{k\geq 0}$ as follows: $\mathbb T^0\mathcal M:=\mathcal M$, if we have defined $\pi_k:\mathbb T^k\mathcal M\rightarrow\mathcal M$, then 
	\begin{equation}\mathbb T^{k+1}\mathcal M:=\mathcal RT_\mathcal M(\mathbb T^k\mathcal M),\end{equation}
and the projection $\pi_{k+1}:\mathbb T^{k+1}\mathcal M\rightarrow\mathcal M$ is obvious.\end{definition}

Notice that, just like $\{T^k\mathcal M\}_{k\geq 1}$, the sequence $\{\mathbb T^k\mathcal M\}_{k\geq 1}$ has a semi-simplicial structure, i.e. for each $k\geq 1$ there are $k$ projections 
	$$\pi_{k,i}:\mathbb T^k\mathcal M\rightarrow\mathbb T^{k-1}\mathcal M,\quad 0\leq i\leq k-1.$$ 
When $i=0$ $\pi_{k,0}$ is the projection of the bundle $\mathcal RT_\mathcal M(\mathbb T^{k-1}\mathcal M)$ on $\mathbb T^{k-1}\mathcal M$, when $i>0$ $\pi_{k,i}$ is obtained from $\pi_{k-i,0}$ by functoriality of $\mathcal RT$. 

Moreover, each $\pi_{k,i}$ is a linear bundle. This is a general fact: consider a morphism of locally trivial bundles
	$$\xymatrix{\mathcal M_1\ar[rd]_{\pi_1}\ar[rr]^F && \mathcal M_2\ar[ld]^{\pi_2}\\ & \mathcal B &}$$
s.t. $F$ is a linear bundle. By functoriality of $\mathcal RT$ we have the diagram
	\begin{equation}\xymatrix{\mathcal RT_\mathcal B\mathcal M_1\ar[rd]_{\mathcal RT_\mathcal B(\pi_1)}\ar[rr]^{\mathcal RT_\mathcal B(F)} && \mathcal RT_\mathcal B\mathcal N\ar[ld]^{\mathcal RT_\mathcal B(\pi_2)}\\ & \mathcal RT_		\mathcal B\mathcal B &}\end{equation}
It is easy to see that fibers of $\mathcal RT_\mathcal B(\pi_1)$ are the tangent bundles to the fibers of $\pi_1$, and similarly for $\pi_2$, and therefore fibers of $\mathcal RT_\mathcal B(F)$ are the same as fibers of $T(F)$. It is well known (e.g. \cite{MK05}) that the latter is a linear bundle, when $F$ is.

\smallskip

Let $\mathbb X^0:=\mathfrak X^0(\mathcal M)$, and let $\mathbb X^k$ be the set of sections of $\mathbb T^k\mathcal M\rightarrow\mathcal M$. Just like with $\mathfrak X^k$, using functoriality of $\mathcal RT$ with respect to local embeddings (proposition \ref{Embeddings}), we obtain a decomposition of each $\nu\in\mathbb X^k$ into a set $\{\alpha_\phi\}$, where each $\alpha\in\mathbb X^1$, and $\phi$ runs over all non-empty subsets of $\{1,\ldots,k\}$.

As with $\mathfrak X^k$, for $\nu\in\mathbb X^k$ the decomposition into $\{\alpha_\phi\}$ depends on the order on the set of projections of $\nu$ on $\mathbb X^1$. Also here we have an action of the symmetric group $\mathbb S_k$ given as in (\ref{TheAction}), but with the Lie bracket substituted by the free bracket $\llbracket-,-\rrbracket$ on $\mathbb X^1=\mathcal RT(\mathfrak X^1)$.

It is clear how to extend additions, strong differences, the cup product, and the composition product from $\{\mathfrak X^k\}$ to $\{\mathbb X^k\}$, and we will freely use the notation of section \ref{kVectors}. We will call $\mathbb X^*:=\{\mathbb X^k\}$ together with these (and other) operations {\bf the full infinitesimal groupoid} of $\mathcal M$. 

\smallskip

In addition to the operations listed above, $\mathbb X^*$ has {\bf homotopy operations}, defined as maps $\{\mathbb X^k\overset{h^k_{ij}}\rightarrow\mathbb X^{k-1}\}_{0\leq i<j\leq k-1}$ for each $k\geq 2$.

To define the homotopy operations we notice that $\mathbb X^1$ has actually two Lie structures. They come from two Lie structures on $\mathcal L(\mathfrak X^1)$. The first one is the free bracket $\llbracket-,-\rrbracket$ given by $\mathcal L$, and the other is the Lie bracket $[-,-]$ on $\mathfrak X^1$, extended to $\mathcal L(\mathfrak X^1)$ by the requirement that 
	$$[x,\llbracket y,z\rrbracket]=\llbracket[x,y],z\rrbracket+\llbracket y,[x,z]\rrbracket.$$
It is easy to check that the kernel of $\mathcal L(\mathfrak X^1)\rightarrow\mathcal R(\mathfrak X^1)$ is a Lie ideal also for $[-,-]$, and hence $\mathbb X^1$ inherits $[-,-]$.

\smallskip

Now, having an additional bracket on $\mathbb X^1$, we have an additional action of $\mathbb S_k$ on $\mathbb X^k$, written in terms of $\{\alpha_\phi\}$, $\alpha_\phi\in\mathbb X^1$. That is, we apply the same formula (\ref{TheAction}), but instead of the free bracket we use $[-,-]$.

For $k\geq 2$ let $\nu\in\mathbb X^k$, and let $\sigma_{ij}\in\mathbb S_k$ be the swapping of $i$ and $j$. Let $\mu\in\mathbb X^{k-2}$ be the projection of $\nu$ on the $\phi$-face, where $\phi=\{0,\ldots,\widehat{i},\ldots,\widehat{j},\ldots,k-1\}$  (if $k=2$ $\mu$ is just $\mathcal M$ itself). Clearly $\mu$ is also projection of $\sigma_{ij}(\nu),\sigma_{ij}'(\nu)$, where $\sigma_{ij}$ acts using $\llbracket-,-\rrbracket$, and $\sigma'_{ij}$ acts using $[-,-]$.

As $2$-vector fields on the image of $\mu$ in $\mathbb T^{k-2}\mathcal M$, $\sigma_{ij}(\nu),\sigma_{ij}'(\nu)$ have the same boundaries, and we can take their strong difference. It is an element of $\mathbb X^{k-1}$, and we define $h^k_{ij}(\nu)$ to be this element. A straightforward but long computation shows that modulo homotopies $h^{<k}_{**}$, $h^k_{ij}$ is well defined with respect to the actions of symmetric groups, i.e. for any $\sigma\in\mathbb S_k$ there is $\tau\in\mathbb S_{k-1}$ s.t. 
	$$h^k_{\sigma(i)\sigma(j)}(\sigma\nu)\sim\tau(h^k_{ij}(\nu)).$$
	
\smallskip

The equivalence relation on $\mathbb X^{k-1}$, defined by $h^k_{ij}(\nu)$ is the following: for $\xi,\xi'\in\mathbb X^{k-1}$ $\xi\sim\xi'$ if $\xi-_\phi\xi'=h^k_{ij}(\nu)$. Consequently, it is natural to say that $\nu\in\mathbb X^k$ defines {\bf trivial homotopies}, if $\sigma_{ij}(\nu)=\sigma'_{ij}(\nu)$ for any $0\leq i<j\leq k-1$, i.e. if $h^k_{ij}(\nu)$ consists of trivial vectors on the $\phi$-face of $\nu$.

\smallskip

What do we get if we take the subset of $\mathbb X^*$, consisting of elements, that define trivial homotopies, and divide it by the homotopy equivalence? We denote the result by $\mathcal H(\mathbb X^*)$ and call it {\bf the cohomology of $\mathbb X^*$.}

\smallskip

First of all it is clear that any $\nu\in\mathbb X^k$ is equivalent to some $\nu'\in\mathfrak X^k$, and no two $\nu',\nu''\in\mathfrak X^k$ are equivalent. 

Secondly, a $\{\alpha_\phi\}=\nu\in\mathfrak X^k$ defines trivial homotopies if and only if for any $\phi,\psi\subseteq\{0,\ldots,k-1\}$, s.t. $\phi\cap\psi=\emptyset$ we have that either $\alpha_\phi=\alpha_\psi$ or at least one of $\alpha_\phi,\alpha_\psi$ is $0$. This is just the condition for $\llbracket\alpha_\phi,\alpha_\psi\rrbracket=0$. 

There are quite many such $\nu$'s, but not that many if we discard degenerate $k$-submanifolds and divide by reparametrizations. Note, that taking the canonical zero section, we can consider any $\alpha\in\mathfrak X^1$ as a section of $T^2\mathcal M$, i.e. we take $\alpha\cup 0$, and this is obviously a degenerate $2$-vector field. Also $\alpha\cup\alpha$ is degenerate, since it is the jet of a $1$-dimensional submanifold. Together we have 
	$$\alpha\sim\alpha\cup 0\sim\alpha\cup\alpha\sim 0\cup\alpha,$$
In addition, for $k\geq 2$ we have reparametrizations of $k$-vectors, i.e. maps $\mathcal W_k\rightarrow\mathcal W_k$, that have $1$ as their Jacobian. 

Discarding degenerate fields implies that any $\nu\in\mathfrak X^*$, that defines trivial homotopies, is equivalent to some $\nu'=\{\alpha'_\phi\}$, s.t. $\alpha'_\phi=0$, unless $\phi=\{0,\ldots,i\}$ for some $i\leq k-1$. Dividing by reparametrizations means that $\nu'$ is linear over $\mathfrak X^0$ in each one of $\alpha'_\phi$'s, and discarding degenerate fields again we get that $\nu'$ is an element of some alternating power of $\mathfrak X^1$ over $\mathfrak X^0$, i.e. $\mathcal H(\mathbb X^*)$ is the set of decomposable elements of $\wedge^*\mathfrak X^1$. Taking the $\mathfrak X^0$-module generated by $\mathcal H(\mathbb X^*)$ we get all of $\wedge^*\mathfrak X^1$.

\smallskip

Finally we discuss algebraic operations on $\mathcal H(\mathbb X^*)$. On $\mathbb X^*$ we have some additions, cup product, and the composition product. Additions translate to the addition on $\wedge^*\mathfrak X^1$, and the cup product becomes the wedge product. 

With the composition product it is not as simple. Recall how one defines the Lie derivative of a vector field $\alpha$ along another vector field $\beta$: one takes the composition product $\alpha\times\beta\in\mathfrak X^2$, uses trivialization of $T\mathcal M$ over the integral curves of $\beta$ to find the projection of $\alpha\times\beta$ on the fibers of $T\mathcal M$, and then uses the linear structure on these fibers to identify tangents to fibers with their points. The resulting section of $T\mathcal M$ is $[\beta,\alpha]$.

In terms of higher categories this is what is called {\bf a thin structure}, i.e. $\alpha\times\beta$ is not the composition, but one of its faces is. Since $\mathcal H(\mathbb X^*)$ is not additive, but multi-linear, we need not one face but many, and the resulting operation is the Schouten bracket.\footnote{Composition product is also a thin structure on all of $\mathbb X^*$, but its algebra of faces is beyond the scope of this paper.}

\smallskip

We would like to stress that $\wedge^*\mathfrak X^1$, obtained as above from $\mathcal H(\mathbb X^*)$, is not a Gerstenhaber algebra. Elements of $\wedge^k\mathfrak X^1$ represent $k$-morphisms, and hence sit in degree $-k+1$, if we use cohomological notation. Therefore, while the bracket is of degree $0$, the cup product is of degree $-1$.


\begin{thebibliography}{Bour}
\bibitem[Be08]{Be08} W.Bertram. {\it Differential geometry over general base fields and rings.}  Mem. Amer. Math. Soc.  192,  no. 900, x+202 pp. (2008).
\bibitem[Ka07]{Ka07} M.Kapranov. {\it Free Lie algebroids and the space of paths.} Sel. math., New ser. 13, pp. 277-319 (2007).
\bibitem[KL84]{KL84} A.Kock, R.Lavendhomme. {\it Strong infinitesimal linearity, with applications to strong difference and affine connections.} Cahiers de Topologie et Geom. Diff., 25, pp. 311-324 (1984).
\bibitem[MK05]{MK05} K.C.H.Mackenzie. {\it General theory of Lie groupoids and Lie algebroids.} London Mathematical Society, Lecture Note Series 213 (2005), 539 pages.
\bibitem[MR91]{MR91} I.Moerdijk, G.E.Reyes. {\it Models for smooth infinitesimal analysis.} Springer (1991), 399+X pages.
\bibitem[Wh82]{Wh82} J.E.White. {\it The method of iterated tangents with applications in local Riemannian geometry.} Pitman, Monographs and Studies in Mathematics 13 (1982), 272 pages.

\end{thebibliography}
\end{document}